\documentclass{amsart}
\usepackage[dvipdfm]{graphicx}
\usepackage{verbatim}
\usepackage{amssymb}
\usepackage{amsbsy}
\usepackage{amscd}
\usepackage{amsmath}
\usepackage{amsthm}
\usepackage{amsxtra}
\usepackage{latexsym}
\usepackage[mathscr]{eucal}
\usepackage{tikz}
\usetikzlibrary{arrows}

\newcommand{\Irr}{\on{Irr}}

\newcommand{\BRS}[2]{H^{\semiinf+0}_{#1}(#2)}

\newcommand{\wh}{\widehat}

\newcommand{\mf}{\mathfrak}

\newcommand{\on}{\operatorname}

\newcommand{\Vg}[1]{V^{#1}(\fing)}

\newcommand{\Prp}{Adm_+}

\newcommand{\isomap}{{\;\stackrel{_\sim}{\to}\;}}

\newcommand{\W}{\mathscr{W}}

\newcommand{\nc}{\newcommand}
\nc{\Hp}[1]{H^{#1}}

\newcommand{\affh}{\widehat{\mathfrak{h}}}

\newcommand{\affg}{\widehat{\mathfrak{g}}}

\newcommand{\fing}{\mathfrak{g}}
\newcommand{\finh}{\mathfrak{h}}

\newcommand{\tp}{{\mathrm{top}}}

\renewcommand{\Irr}[1]{\mathbf{L}_{#1}}

\newcommand{\BGG}{{\mathcal O}}

\newcommand{\1}{{\mathbf{1}}}
\newcommand{\teigi}{\stackrel{\mathrm{def}}{=}}

\newcommand{\dual}[1]{{#1}^*}

\newcommand{\ket}{{\rangle}}

\newcommand{\Lam}{\Lambda}

\newcommand{\lam}{\lambda}
\newcommand{\ra}{\rightarrow}
\newcommand{\+}{\mathop{\oplus}}
\newcommand{\Z}{\mathbb{Z}}
\newcommand{\Mod}{\text{-}\mathrm{Mod}}

\newcommand{\inv}{^{-1}}

\newcommand{\C}{\mathbb{C}}
\newcommand{\che}{^{\vee}}

\theoremstyle{plain}
\newtheorem{Th}{Theorem}[section]
\newtheorem*{MainTh}{Main Theorem}
\newtheorem{Pro}[Th]{Proposition}

\theoremstyle{definition}

\theoremstyle{remark}

\newtheorem{Rem}[Th]{Remark}

\newcommand{\semiinf}{\frac{\infty}{2}}

\DeclareMathOperator{\tr}{tr}

\title{
Rationality of 
Bershadsky-Polyakov vertex algebras}
\author{Tomoyuki Arakawa}
\address{Research Institute for Mathematical Sciences, Kyoto University,
 Kyoto 606-8502 JAPAN}

\email{arakawa@kurims.kyoto-u.ac.jp}

\thanks{This work is partially  supported 
by the JSPS Grant-in-Aid  for Scientific Research (B)
No.\ 20340007}

\begin{document}
\maketitle

\begin{abstract}
We prove the conjecture of Kac-Wakimoto 
on the rationality 
of exceptional $W$-algebras
for the first non-trivial
series,
namely, for the
 Bershadsky-Polyakov vertex algebras
$W_3^{(2)}$
 at level
$k=p/2-3$ with $p=3,5,7,9,\dots$.
This gives new examples of rational conformal field theories.
\end{abstract}

\section{Introduction}
Recently,
a remarkable family of 
$W$-algebras associated with  simple Lie algebras
and their {\em non-principal} nilpotent elements,
called {\em exceptional $W$-algebras},
 has been discovered by Kac
and Wakimoto \cite{KacWak08}.
In \cite{KacWak08} it was  conjectured that 
with an exceptional $W$-algebra one can associate a rational
conformal field theory.

As a first step to resolve  the Kac-Wakimoto conjecture
we have proved in the previous article  \cite{Ara09b}
that 
exceptional $W$-algebras are {\em lisse},
or equivalently
\cite{Ara12}, $C_2$-cofinite.
Therefore
it remains \cite{Zhu96,Hua08} to show that
exceptional $W$-algebras are {\em rational},
i.e., that the representations are completely reducible,
 in order to prove the
 conjecture.
In this article we prove the rationality of
the first non-trivial series of exceptional $W$-algebras,
that is,
the 
{\em Bershadsky-Polyakov (vertex) algebras}
 $W_3^{(2)}$ \cite{Pol90,Ber91}
at level
$k=p/2-3$ with $p=3,5,7,9,\dots$.
The  vertex algebra $W_3^{(2)}$
is 
the $W$-algebra associated with
$\fing=\mf{sl}_3$ and it minimal nilpotent element.

Let us state our main result more precisely:
Let  $\W_k$
denote
the unique simple quotient of  
$W_3^{(2)}$
at level $k\ne -3$.
\begin{MainTh}[Conjectured by Kac and Wakimoto \cite{KacWak08}]
Let $p$ be an odd integer equal or greater than $3$,
$k=p/2-3$.
Then 
the vertex algebra
$\W_k$ is rational.
The simple $\W_k$-modules 
are parameterized by the set 
of integral dominant weights of $\widehat{\mf{sl}}_3$
of level $p-3$.
These simple modules can be obtained by the 
quantum BRST reduction from irreducible admissible representations of
 $\widehat{\mf{sl}}_3$ of level $k$.
\end{MainTh}

For
$p=3$,
$\W_{3/2-3}$ is one-dimensional.
In the remaining cases
$\W_{p/2-3}$ are conformal with negative central charges.

\smallskip
We note that
Zhu's  algebra  of $W_3^{(2)}$
is  closely related with 
Smith's algebra \cite{Smi90}
which is a deformation
of the universal enveloping algebra
$U(\mf{sl}_2(\C))$ of $\mf{sl}_2(\C)$, and
that
the rational quotient $\W_{p/2-3}$
has  features
in common with the $\wh{\mf{sl}}_2$-integrable affine vertex algebras
in the sense that
 the following relations hold:
 \begin{align*}
  :G^+(z)^{p-2}:=:G^-(z)^{p-2}:=0,
 \end{align*}
where 
$G^+(z)$ and $G^-(z)$  are the
standard generating fields
of $\W_{p/2-3}$,
see below.

\section{Bershadsky-Polyakov algebras at exceptional levels.}
Let $\W^k$ denote the 
Bershadsky-Polyakov (vertex)
algebra 
$W_3^{(2)}$ at level $k\ne -3$,
which
 is the vertex algebra freely generated by
the fields $J(z), G^{\pm}(z), T(z)$
with the following OPE's:
\begin{align*}
 &J(z)J(w)\sim \frac{2k +3}{3(z-w)^2},\quad G^{\pm}(z)G^{\pm}(w)\sim
 0,\\
&J(z)G^{\pm}(w)\sim \pm \frac{1}{z-w}G^{\pm}(w),\\
&T(z)T(w)\sim
- \frac{(2k+3)(3k+1)}{2(k+3)(z-w)^4}+\frac{2}{(z-w)^2}T(w)+\frac{1}{z-w}\partial
T(w),\\
&T(z)G^{\pm}(w)\sim
 \frac{3}{2(z-w)^2}G^{\pm}(w)+\frac{1}{z-w}\partial G^{\pm}(w),\\
& T(z)J(w)\sim \frac{1}{(z-w)^2}J(w)+\frac{1}{z-w}\partial J(w),\\
&G^{+}(z)G^-(w)\sim \frac{(k+1)(2k+3)}{(z-w)^3}
+\frac{3(k+1)}{(z-w)^2}J(w)\\
&\qquad\qquad\qquad\qquad
+\frac{1}{z-w}\left(3:J(w)^2:+\frac{3(k+1)}{2}\partial J(w)
-(k+3)T(w)\right).
\end{align*}

As in introduction we denote by $\W_k$ the unique simple quotient of $\W^k$.
\begin{Th}[\cite{Ara09b}]\label{Th:lisse}
  Let $k$, $p$ be as in Main Theorem.
Then $\W_k$ is lisse,
or equivalently,
$C_2$-cofinite.
\end{Th}

Set
\begin{align*}
 L(z)=\sum_{n\in \Z}L_{n}z^{-n-2}=T(z)+\frac{1}{2}\partial J(w).
\end{align*}
This defines a conformal vector
of $\W^k$
with central charge
\begin{align*}
 c(k)=-\frac{4 (k+1) (2 k+3) }{k+3}=
-\frac{4 (p-4) (p-3)}{p},
\end{align*}
which gives 
$J$, $G^+$, $G^-$ 
 conformal weights
$1$, $1$, and $2$,
respectively.
Hence $\W^k$ is $\Z_{\geq 0}$-graded  with respect to
the Hamiltonian  $L_0$.
We expand the corresponding fields accordingly:
\begin{align*}
 &J(z)=\sum_{n\in \Z}J_n z^{-n-1},\quad G^+(z)=\sum_{n\in \Z}G^+_n z^{-n-1}
,\quad
&G^-(z)=\sum_{n\in \Z}G^-_n z^{-n-2}.
\end{align*}
We  have
\begin{align*}
&[J_m,J_n]=\frac{2k+3}{3}m\delta_{m+n,0},\quad [J_m,G_n]=G_{m+n},\quad
[J_m,F_n]=-F_{m+n},\\
&[L_m,J_n]=-nJ_{m+n}-\frac{(2k+3)(m+1)m}{6}\delta_{m+n,0},\\
&[L_m,G_n^+]=-nG_{m+n}^+,\quad 
[L_m,G^-_n]=(m-n)G^-_{m+n},\\
&[G^+_m,G^-_n]
 =3(J^2)_{m+n}+
\left(3(k+1)m  -(2 k+3)(m+n+1)
\right)J_{m+n}
-(k+3)L_{m+n}\\
&\qquad \qquad \qquad +\frac{(k+1)(2 k+3)m(m+1)}{2}\delta_{m+n,0},
\end{align*}
where $\sum_{n\in \Z}(J^2)_n z^{-n-2}\teigi :J(z)^2:$.

For $(\xi,\chi)\in \C^2$,
let $L(\xi,\chi)$ be the irreducible representation of 
$\W^k$ generated by the vector 
$|\xi,\chi\ket$
such that
\begin{align*}
&J_0|\xi,\chi\ket=\xi |\xi,\chi\ket,\quad
J_n|\xi,\chi\ket=0\quad \text{for }n>0,\\
&
 L_0 |\xi,\chi\ket=\chi |\xi,\lam\ket,
\quad L_n |\xi,\chi\ket =0\quad \text{for }n>0,
\\
&G^-_{n}|\xi,\chi\ket=0\quad \text{for }n\geq 0,\quad
G^+_{n}|\xi,\chi\ket=0\quad \text{for }n\geq 1.
\end{align*}
By Theorem \ref{Th:lisse},
any simple $\W_k$-module is of the
form $L(\xi,\lam)$ with some $\xi$ and $\chi$.
(It is important that the lisse condition is defined independent of the 
choice of a conformal vector.)

For a 
$\W^k$-module $M$
set 
\begin{align*}
M_{a,d}=\{m\in M;
J_0m=a m,\ L_0m=d m\}.
\end{align*}It is clear that
$L(\xi,\chi)=\bigoplus\limits_{(a,d)\in \C^2
\atop d\in \chi+\Z_{\geq 0}}L(\xi,\chi)_{a,d}$,
$\dim L(\xi,\chi)_{\xi,\chi}=1$.
Let
\begin{align*}
L(\xi,\chi)_{\tp}
=\{v\in L(\xi,\chi);
L_0v=\chi v\}=\bigoplus_{a}L(\xi,\chi)_{a,\chi}.
\end{align*}By definition
$L(\xi,\chi)_{\tp}$ is spanned by the vectors
$(G^+_0)^i|\xi,\chi\ket$ with $i\geq 0$.

Following \cite{Smi90}
set
\begin{align*}
 g(\xi,\chi)
=-(3 \xi^2-(2k+3)\xi-(k+3)\chi),
\end{align*}
so that $G^-_0 G^+_0|\xi,\chi\ket=g(\xi,\chi)|\xi,\chi\ket$.
We have
\begin{align*}
 G^-_0  (G^+_0)^i|\xi,\chi\ket=i h_i(\xi,\chi)(G^+_0)^{i-1}|\xi,\chi\ket,
\end{align*}
where
\begin{align*}
 h_i(\xi,\chi)&=\frac{1}{i}(g(\xi,\chi)+
g(\xi+1,\chi)+\dots +g(\xi+i-1,\chi))\\
&=-i^2+k i-3 \xi  i+3 i-3 \xi ^2-k+2 k \xi +6 \xi +k \chi +3
   \chi -2.
\end{align*}
Hence we have the following assertion.
\begin{Pro}
If
the space $L(\xi,\chi)_{\tp}$
 is $n$-dimensional,
then
$h_n(\xi,\chi)=0$.
\end{Pro}

Define
\begin{align*}
 \Delta(-J,z)=
z^{-J_0}\exp\left(\sum_{k=1}^{\infty}(-1)^{k+1}\frac{-J_k}{k z^k}\right),
\end{align*}
and set
\begin{align*}
 \sum_{n\in \Z}\psi(a_{(n)})z^{-n-1}=Y(\Delta(-J,z)a,z)
\end{align*}
for $a\in \W^k$.
For any $\W^k$-module $M$,
we can define on $M$ a new $\W^k$-module
structure
by twisting the action of $\W^k$
as
$a_{(n)}\mapsto \psi(a_{(n)})$ (\cite{Li97}).
We denote by 
$\psi(M)$ thus obtained
 $\W^k$-module 
from $M$.
\begin{Pro}\label{Pro:psi(L)}
 Suppose that $\dim L(\xi,\chi)_{\tp}=i$.  
Then
\begin{align*}
\psi(L(\xi,\chi))\cong L(\xi+i-1-\frac{2 k+3}{3},\chi
-(\xi-i+1)+\frac{2k+3}{3}).
\end{align*}
\end{Pro}
\begin{proof}
 The assertion follows from
the fact that
\begin{align*}
\psi(J_n)=J_n-\frac{2k+3}{3}\delta_{n,0},\quad
\psi(L_n)=L_n-J_n+\frac{2k+3}{3},\\
 \psi(G_n^+)=G^+_{n-1},\quad
\psi(G^-_n)=G^-_{n+1}.
\end{align*}
\end{proof}
By solving the equation
\begin{align*}
h_i(\xi,\chi)=h_j(\xi+i-1-\frac{2k+3}{3},\chi-(\xi-i+1)+\frac{2k+3}{3})
\end{align*}
we obtain the following assertion.
\begin{Pro}\label{Pro:top-and-psi-top}
Suppose that
 $\dim L(\xi,\chi)_{\tp}=i$
and   $\dim \psi(L(\xi,\chi))_{\tp}=j$.
Then
\begin{align*}
&\xi=\xi_{i,j}\teigi \frac{1}{3} (-2 i-j+2 k+6)
 ,\\
&\chi=\chi_{i,j}\teigi \frac{i^2+j i-k i-3 i+j^2-6 j-2 j k+3 k+6}{3 (k+3)}.
\end{align*}
\end{Pro}
\begin{Pro}\label{Pro:crutial}
 Let $k$, $p$ be as in Main Theorem.
Then $(G^+_{-1})^{p-2}\1$
belongs to the maximal ideal of
$\W^k$.
\end{Pro}
\begin{proof}
Since
$\xi_{1,p-2}=\chi_{1,p-2}=0$,
the correspondence  $\1\mapsto |\xi_{1,p-2},\chi_{1,p-2}\ket$
gives an isomorphism
 $\W_k\cong L(\xi_{1,p-2},\chi_{1,p-2})$.
Because 
\begin{align*}
h_{p-2}(\xi_{1,p-2}-(2k+3)/2,\chi_{1,p-2}+(2k+3)/3)=0,
\end{align*}
from Proposition \ref{Pro:psi(L)}  it follows that
$\psi(\W_k)_{\tp}$ is at most $p-2$-dimensional.
Hence $(G_{-1}^+)^{p-2}\1=0$.
\end{proof}
\begin{Rem}
One can  show that
in fact $(G_{-1}^+)^{p-2}$ generates
the maximal ideal of $\W^k$.
However we do not need this fact.
\end{Rem}

\begin{Pro}\label{Pro:simples}
 Let $k$, $p$ be as in Main Theorem.
Then
any simple $\W_k$-module is isomorphic to
$L(\xi_{i,j},\chi_{i,j})$
for some $(i,j)$
such that
 $1\leq i\leq p-2$,
$1\leq j\leq p-i-1$.
\end{Pro}
\begin{proof}
 Let $L(\xi,\chi)$ be  a simple
$\W_k$-module.
As
$:G^+(z)^{p-2}:=0$
on $L(\xi,\chi)$ 
by Proposition \ref{Pro:crutial},
$L(\xi,\chi)_{\tp}$ is at most $(p-2)$-dimensional.
Since $\psi(L(\xi,\chi))$ is also 
a $\W_k$-module
we have
$(\xi,\chi)=(\xi_{i,j},\chi_{i,j})$ for some
$1\leq i,j\leq p-2$.
Because 
$\psi(\psi(L(\xi_{i.j},\chi_{i,j})))$ is also
a $\W_k$-module it follows that
$\xi_{i,j}+i-1-\frac{2k+3}{3}=\frac{i-j}{3}\leq \frac{-2j-1+2k+6}{3}
=\frac{p-2j-1}{3}$.
Hence $j\leq  p-i-1$.
\end{proof}

 The simple $\W^k$-modules
$L(\xi_{i,j},\chi_{i,j})$ with $1\leq i\leq p-2$,
$1\leq j\leq p-i-1$,
are  mutually non-isomorphic since
their highest weights are distinct.

\section{Proof of Main Theorem}
Let
$k$,
$p$ be as in Main Theorem.

Let
$\fing=\mf{sl}_3$
 as in introduction,
$\finh\subset \fing$ be the Cartan subalgebra of $\fing$
consisting of diagonal matrixes.
Set  $h_i=E_{i,i}-E_{i+1, i+1}$,
$h_{\theta}=h_1+h_2$,
$e_i=e_{\alpha_i}=E_{i, i+1}$,
$f_i=f_{\alpha_i}=E_{i+1,i}$ for $i=1,2$,
$e_{\theta}=E_{1,3}$,
$f_{\theta}=E_{3,1}$,
where $E_{i,j}$ is the matrix element.
We equip $\fing$ the invariant form
$(x|y)=\tr(xy)$.
Set $\bar \Lam_1=(2 h_1+h_2)/3$,
$\bar \Lam_1=( h_1+2h_2)/3$,
so that
$(\bar \Lam_i|h_j)=\delta_{i,j}$.

Let $\affg=\fing[t,t\inv]\+ \C K\+ \C D$
be the 
(non-twisted) affine Kac-Moody algebra associated with
$\fing$,
where $K$ is the central element and $D$ is the degree operator.
Let
$\affh=\finh\+ \C K \+ \C D \subset \affg$ the standard Cartan
subalgebra,
$\dual{\affh}=\finh^* \+ \C \Lam_0\+ \C \delta$
the dual of $\affh$,
where $\Lam_0$ and $\delta$
are elements dual to $K$ and $D$,
respectively.

The
vector
$f_{\theta}$
is 
a the minimal nilpotent  element of $\fing$.
Let $\fing=\bigoplus_{j\in \frac{1}{2}\Z}\fing_j$
be
the corresponding Dynkin grading:
$\fing_j=\{u\in \fing;[h_{\theta},u]=2j u\}$.
Denote by $\BRS{f_{\theta}}{?}$
the BRST cohomology 
of the generalized quantized Drinfeld-Sokolov reduction
associated with $(\fing,f_{\theta})$ and the 
Dynkin grading.
We have  \cite{KacRoaWak03,KacWak04}  
the vertex algebra isomorphism 
\begin{align*}
 \W^k\isomap \BRS{f_{\theta}}{\Vg{k}},
\end{align*}
which is
given by the following assignment:
\begin{align*}
&J(z)\mapsto J^{-\bar \Lam_1+\bar \Lam_2}(z)-:\Phi_1(z)\Phi_2(z):,\\
&
 G^+(z)\mapsto J^{f_1}(z)-:J^{h_1}(z)\Phi_2(z):+:\Phi_1(z)\Phi_2(z)^2:-
(k+1)\partial \Phi_2(z),\\
&
 G^+(z)\mapsto -J^{f_2}(z)-:J^{h_2}(z)\Phi_1(z):-:\Phi_1(z)^2\Phi_2(z):-
(k+1)\partial \Phi_1(z),\end{align*}
Here
\begin{align*}
J^u(z)=u(z)-\sum\limits_{\beta,\gamma
\in
\{\alpha_1,\alpha_2,\theta\}}c_{u,f_{\beta}}^{f_{\gamma}}:\psi_{\beta}^*(z)\psi_{\gamma}(z):
\end{align*}for $u\in \fing$,
$c_{u_1,u_2}^{u_3}$ is the structure constant,
$\psi_{\alpha}(z)$, $\psi_{\alpha}^*(z)$
with $\alpha\in \{\alpha_1,\alpha_2,\theta\}$
are fermionic ghosts satisfying 
\begin{align}
 \psi_{\alpha}(z)\psi_{\beta}^*(w)\sim
 \frac{\delta_{\alpha,\beta}}{z-w},
\quad  \psi_{\alpha}(z)\psi_{\beta}(w)\sim
 \psi_{\alpha}^*(z)\psi_{\beta}(^*w)\sim 0,
\label{eq:fermion}
\end{align}
$\Phi_1(z)$,
$\Phi_2(z)$ are bosonic ghosts satisfying
 \begin{align*}
  \Phi_1(z)\Phi_2(w)\sim \frac{1}{z-w},\quad
\Phi_i(z)\Phi_i(w)\sim 0,
 \end{align*}
and the BRST differential is the zero mode of the field
\begin{align*}
 Q(z)=\sum_{\alpha\in \{\alpha_1,\alpha_2,\theta\}}
e_{\alpha}(z)\psi_{\alpha}^*(w)
-:\psi_{\alpha_1}^*(z)\psi_{\alpha_2}^*(z)\psi_{\theta}(z):\\
+\Phi_1(z)\psi_{\alpha_1}^*(z)+\Phi_2(z)\psi_{\alpha_2}(z)
+\psi_{\theta}(z).
\end{align*}

Let $\BGG_k$ be the category $\BGG$ of $\affg$ at level $k$,
$\Irr{\lam}$ the irreducible representation of $\affg$
with highest weight $\lam$.
Denote by  $\W^k\Mod$  the category of $\W^k$-modules.
\begin{Th}[\cite{Ara05}]\label{Th:Arakawa2}$ $

 \begin{enumerate}
  \item The functor 
$\BRS{f_{\theta}}{?}:\BGG_k\ra \W^k\Mod$,
$M\mapsto \BRS{f_{\theta}}{M}$,
is exact.
\item For $\lam\in \dual{\affh}$
we have
$\BRS{f_{\theta}}{\Irr{\lam}}=0$
 if and only if $\lam(\alpha_0\che)\in
      \{0,1,2,3,\dots\}$.
Otherwise $\BRS{f_{\theta}}{\Irr{\lam}}$ is irreducible.
 \end{enumerate}
\end{Th}

Let $Adm^k$ be the set of admissible weights \cite{KacWak89}
of $\affg$ of level $k$,
and put
\begin{align*}
 \Prp^k=\{\lam\in Adm^k; \bar \lam\text{ is an integral dominant weight
of $\fing$}\},
\end{align*}
where $\dual{\affh}\ni \lam\mapsto \bar \lam\in \dual{\finh}$
is the restriction.
Then
\begin{align*}
 \Prp^k
=\{\bar \mu +k \Lam_0;\mu\in \wh P_{++}^{p-3}
\},
\end{align*} 
where 
$\wh P_{++}^{p-3}$ is the set of integral dominant weights of $\affg$
of level $p-3$.
Explicitly,
we have
\begin{align*}
 \Prp^k=\{\lam_{i,j}; 1\leq i\leq p-2,\
1\leq j\leq p-i-1\},
\end{align*}
where
\begin{align*}
 \lam_{i,j}=(i-1)\bar \Lam_1+(p-i-j-1)\bar \Lam_2+k\Lam_0.
\end{align*}
Note that
\begin{align}
 \xi_{i,j}=(\lam_{i,j}|-\bar\Lam_1+\bar \Lam_2),
\quad
\chi_{i,j}=\frac{(\lam_{i,j}|\lam_{i,j}+2\bar\rho)}{2(k+3)}-(\lam_{i,j}|\bar
 \Lam_2),
\label{eq:hw}
\end{align}
where $\bar \rho=\bar \Lam_1+\bar \Lam_2$.

Recall the following  result of Malikov and Frenkel \cite{MalFre99}.  
\begin{Th}[{\cite[Corollary 5.2.2]{MalFre99}}]
\label{Th;Malikov-Frenkel}
For $\lam\in \Prp^k$,
$\Irr{\lam}$ is a module over 
 $\Irr{k\Lam_0}$.
\end{Th}
\begin{Pro}\label{Pro:modules}
For $\lam_{i,j}\in \Prp^k$,
$\BRS{f_{\theta}}{\Irr{\lam_{i,j}}}$  is a simple $\W_k$-module
isomorphic to $L(\xi_{i,j},\chi_{i,j})$.
\end{Pro}
\begin{proof}
By Theorem   \ref{Th:Arakawa2}
we have $\W_k\cong \BRS{f_{\theta}}{\Irr{k\Lam_0}}$.
Hence 
by the functoriality of $\BRS{f_{\theta}}{?}$,
Theorem \ref{Th;Malikov-Frenkel} immediately gives that
$\BRS{f_{\theta}}{\Irr{\lam_{i,j}}}$
  is a  module over $\W_k$.
By Theorem   \ref{Th:Arakawa2},
$\BRS{f_{\theta}}{\Irr{\lam_{i,j}}}$
is (nonzero and) irreducible.
Let $v$ be the image of the 
highest weight vector of $\Irr{\lam_{i,j}}$ in $\BRS{f_{\theta}}{\Irr{\lam_{i,j}}}$.
By  (\ref{eq:hw})
and the fact \cite{KacWak04} that
the image of $L(z)$ in $\W^k$ is cohomologous
to 
\begin{align*}
 L_{\fing}(z)+L_{\on{ch}}(z)+L_{\Phi}(z)+
\partial J^{\bar \Lam_2}(z),
\end{align*}
where $L_{\fing}(z)$ is the Sugawara 
operator of $\fing$,
$L_{\on{ch}}(z)=-\sum_{\alpha=\alpha_1,\alpha_2,\theta}:\phi_{\alpha}(z)
\partial \phi_{\alpha}^*(z)$,
$L_{\Phi}(z)=\frac{1}{2}\left(:\Phi_2(z)\partial \Phi_1(z):-
:\partial \Phi_1(z)\Phi_2(z)\right)$,
 it is straightforward to check that
the assignment
$|\xi_{i,j},\chi_{i,j}\ket\mapsto v$
gives a $\W^k$-module homomorphism.
By the irreducibility,
this must be an isomorphism.
\end{proof}

By Propositions
 \ref{Pro:simples}
and \ref{Pro:modules},
the set $\{\BRS{f_{\theta}}{\Irr{\lam}};
\lam\in Adm^k_+\}$
gives the complete set of isomorphism  classes of 
simple $\W_k$-modules.
Therefore
 Main Theorem  now follows immediately from
the following important result of Gorelik and Kac \cite{GorKac0905}.
\begin{Th}[{\cite[Corollary 8.8.9]{GorKac0905}}]
For any $\lam,\mu\in Adm^k$,
we have
\begin{align*}
\on{Ext}^1_{\W^k\Mod}(\BRS{f_{\theta}}{\Irr{\lam}},\BRS{f_{\theta}}{\Irr{\mu}})=0.
\end{align*}
\end{Th}

\bibliographystyle{jalpha}
\bibliography{math}

\begin{thebibliography}{15}

\bibitem[1]{Ara05}
T. Arakawa.
\newblock Representation theory of superconformal algebras and the
  {K}ac-{R}oan-{W}akimoto conjecture.
\newblock {\em Duke Math. J.}, Vol. 130, No.~3, pp. 435--478, 2005.


\bibitem[2]{Ara12}
T. Arakawa.
\newblock A remark on the {$C_2$} cofiniteness condition on vertex algebras.
\newblock {\em Math. Z.}, Vol. 270, No. 1-2, pp. 559--575, 2012.

\bibitem[3]{Ara09b}
T.  Arakawa.
\newblock Associated varieties of modules over {K}ac-{M}oody algebras and
  {$C_2$}-cofiniteness of {W}-algebras.
\newblock {\em preprint},
\newblock arXiv:1004.1554v2.

\bibitem[4]{Ber91}
Michael Bershadsky.
\newblock Conformal field theories via {H}amiltonian reduction.
\newblock {\em Comm. Math. Phys.}, Vol. 139, No.~1, pp. 71--82, 1991.

\bibitem[5]{GorKac0905}
Maria Gorelik and Victor Kac.
\newblock On complete reducibility for infinite-dimensional {L}ie algebras.
\newblock {\em Adv. Math.}, Vol. 226, No.~2, pp. 1911--1972, 2011.

\bibitem[6]{Hua08}
Yi-Zhi Huang.
\newblock Vertex operator algebras and the {V}erlinde conjecture.
\newblock {\em Commun. Contemp. Math.}, Vol.~10, No.~1, pp. 103--154, 2008.

\bibitem[7]{KacRoaWak03}
Victor Kac, Shi-Shyr Roan, and Minoru Wakimoto.
\newblock Quantum reduction for affine superalgebras.
\newblock {\em Comm. Math. Phys.}, Vol. 241, No. 2-3, pp. 307--342, 2003.

\bibitem[8]{KacWak89}
V.~G. Kac and M.~Wakimoto.
\newblock Classification of modular invariant representations of affine
  algebras.
\newblock In {\em Infinite-dimensional Lie algebras and groups
  (Luminy-Marseille, 1988)}, Vol.~7 of {\em Adv. Ser. Math. Phys.}, pp.
  138--177. World Sci. Publ., Teaneck, NJ, 1989.

\bibitem[9]{KacWak04}
Victor~G. Kac and Minoru Wakimoto.
\newblock Quantum reduction and representation theory of superconformal
  algebras.
\newblock {\em Adv. Math.}, Vol. 185, No.~2, pp. 400--458, 2004.

\bibitem[10]{KacWak08}
Victor~G. Kac and Minoru Wakimoto.
\newblock On rationality of {$W$}-algebras.
\newblock {\em Transform. Groups}, Vol.~13, No. 3-4, pp. 671--713, 2008.

\bibitem[11]{Li97}
Haisheng Li.
\newblock The physics superselection principle in vertex operator algebra
  theory.
\newblock {\em J. Algebra}, Vol. 196, No.~2, pp. 436--457, 1997.

\bibitem[12]{MalFre99}
F.~G. Malikov and I.~B. Frenkel{\cprime}.
\newblock Annihilating ideals and tilting functors.
\newblock {\em Funktsional. Anal. i Prilozhen.}, Vol.~33, No.~2, pp. 31--42,
  95, 1999.

\bibitem[13]{Pol90}
A.~M. Polyakov.
\newblock Gauge transformations and diffeomorphisms.
\newblock {\em Internat. J. Modern Phys. A}, Vol.~5, No.~5, pp. 833--842, 1990.

\bibitem[14]{Smi90}
S.~P. Smith.
\newblock A class of algebras similar to the enveloping algebra of {${\rm
  sl}(2)$}.
\newblock {\em Trans. Amer. Math. Soc.}, Vol. 322, No.~1, pp. 285--314, 1990.

\bibitem[15]{Zhu96}
Yongchang Zhu.
\newblock Modular invariance of characters of vertex operator algebras.
\newblock {\em J. Amer. Math. Soc.}, Vol.~9, No.~1, pp. 237--302, 1996.

\end{thebibliography}

\end{document}